\documentclass[12pt,a4paper]{amsart}

\usepackage[utf8]{inputenc}
\usepackage[T1]{fontenc}
\usepackage[english]{babel}
\usepackage{amsmath, amssymb, amsthm, mathtools}
\usepackage{enumitem}
\usepackage[margin=2.5cm]{geometry}
\usepackage{xcolor}
\usepackage{hyperref}
\hypersetup{colorlinks=true, linkcolor=blue, citecolor=blue, urlcolor=blue}

\usepackage[authoryear]{natbib}
\bibpunct{(}{)}{;}{a}{}{,}

\usepackage{setspace}
\theoremstyle{plain}
\newtheorem{theorem}{Theorem}[section]
\newtheorem{proposition}[theorem]{Proposition}

\newtheorem{corollary}[theorem]{Corollary}

\theoremstyle{definition}

\newcommand{\PP}{\mathbb{P}}
\newcommand{\CC}{\mathbb{C}}

\newcommand{\ZZ}{\mathbb{Z}}
\newcommand{\Q}{\mathbb{Q}}
\newcommand{\Pdual}{\check{\PP}^{2}_{\CC}}
\newcommand{\Ptwo}{\PP^{2}_{\CC}}
\newcommand{\Pone}{\PP^{1}_{\CC}}
\newcommand{\Sing}{\operatorname{Sing}}

\newcommand{\Leg}{\operatorname{Leg}}
\newcommand{\Sym}{\operatorname{Sym}}
\newcommand{\Inv}{\mathrm{inv}}
\newcommand{\Tr}{\mathrm{tr}}
\newcommand{\Hcal}{\mathcal{H}}
\newcommand{\Fcal}{\mathcal{F}}
\newcommand{\Wcal}{\mathcal{W}}

\newcommand{\Tcal}{\mathcal{T}}
\newcommand{\GH}{\mathcal{G}_{\mathcal{H}}}
\newcommand{\GF}{\mathcal{G}_{\mathcal{F}}}

\newcommand{\Hwave}{\widetilde{\mathcal{H}}}
\newcommand{\Fwave}{\widetilde{\mathcal{F}}}
\newcommand{\boxprod}{\boxtimes}

\title{Homogeneous pre-foliations of co-degree one and degree four on the projective plane}

\author[C. Pracias]{Carla Pracias}
\address[C. Pracias]{Instituto de Matem\'atica e Estat\'istica, Universidade Federal Fluminense, Rua Professor Marcos Waldemar de Freitas Reis, s/n, CEP 24210-201\\
Bloco H, Campus do Gragoat\'a - Niter\'oi - RJ, Brasil}
\email{carla\_pracias@id.uff.br}

\author[M. Falla Luza]{Maycol Falla Luza}
\address[M. Falla Luza]{Instituto de Matem\'atica e Estat\'istica, Universidade Federal Fluminense, Rua Professor Marcos Waldemar de Freitas Reis, s/n, CEP 24210-201\\
Bloco H, Campus do Gragoat\'a - Niter\'oi - RJ, Brasil}
\email{hfalla@id.uff.br}

\date{}

\begin{document}

\maketitle

\begin{abstract}
We classify, up to projective automorphism, all homogeneous pre-foliations of co-degree $1$ and degree $4$ on the complex projective plane $\Ptwo$ whose Legendre transform defines a flat $4$-web. The classification is organized according to the type of the underlying homogeneous foliation $\Hcal$ of degree~$3$, distinguishing the cases $\deg(\Tcal_{\Hcal})=2$, $3$, and~$4$. The case $\deg(\Tcal_{\Hcal})=2$ was treated by Bedrouni, while the cases $\deg(\Tcal_{\Hcal})=3$ and $\deg(\Tcal_{\Hcal})=4$ are completed here. The proof combines Bedrouni's curvature-holomorphy criteria with explicit normal forms and symbolic computation; the result yields a finite list of explicit one-forms, parametrized by the ramification data of the Gauss map of~$\Hcal$, all displayed explicitly in the article (the few cases whose parameters are roots of higher-degree polynomials being written out in full in Appendix~\ref{app:reproducibility}).
\end{abstract}

\section{Introduction}\label{sec:intro}

Let $\Wcal$ be a $k$-web on a complex surface, that is, the local datum of $k$ pairwise transverse foliations away from a hypersurface called the discriminant. To each such object Blaschke and Dubourdieu associated a meromorphic $2$-form, the \emph{curvature} $K(\Wcal)$, whose vanishing is a defining feature of \emph{flat webs} (see~\citep[Chapters~1 and~2]{PP15} for a modern account). Flatness is a powerful constraint: by Mihaileanu's theorem, it is a necessary condition for $\Wcal$ to have maximal rank, and hence is closely tied to the existence of abelian relations and to the algebraization problem of Lie and Bol (see~\citep[Chapter~6]{PP15}). The set of flat $k$-webs of degree $d$ on the complex projective plane $\Ptwo$ is a Zariski closed subset of the moduli space $\Wcal(k,d)$ of $k$-webs of degree $d$, and one of the central problems in modern web geometry is to classify its irreducible components.

A productive source of webs on $\Ptwo$ is the \emph{Legendre transform}: starting from a foliation $\Fcal$ of degree $d$ on $\Ptwo$, one obtains a $d$-web $\Leg\Fcal$ on the dual plane $\Pdual$ whose leaves are the duals of the leaves of $\Fcal$. The interplay between the geometry of $\Fcal$ and the flatness of $\Leg\Fcal$ has attracted considerable attention. Marín and Pereira~\citep{MP13} proved that the Legendre transform of any reduced convex foliation on $\Ptwo$ is flat, and Bedrouni and Marín~\citep{BM18} classified all homogeneous foliations of degree~$3$ on $\Ptwo$ with flat Legendre transform: there are exactly eleven of them, up to automorphism of $\Ptwo$.

The natural next step is to allow the foliation to be enriched by a curve, i.e.\ to study \emph{pre-foliations}. A pre-foliation of co-degree $k$ and degree $d$ is the formal product $\Fwave = C\boxprod\Fcal$ of a reduced curve $C$ of degree $k$ and a foliation $\Fcal$ of degree $d-k$. The Legendre transform of $\Fwave$ is naturally a $d$-web on $\Pdual$, and Bedrouni~\citep{Bed23} initiated the study of those pre-foliations of co-degree~$1$ for which $\Leg\Fwave$ is flat. He provided, in particular, sharp curvature-holomorphy criteria along the irreducible components of the discriminant, in terms of the ramification data of the Gauss map of the underlying foliation, and he completed the classification in degree $4$ under the assumption $\deg(\Tcal_{\Hcal})=2$ (see~\citep[Section~4]{Bed23}). 

The present paper completes the classification of homogeneous pre-foliations of co-degree~$1$ and degree~$4$ on $\Ptwo$ with flat Legendre transform.

\begin{theorem}[Main Theorem]\label{thm:main}
Let $\Hwave = \ell\boxprod\Hcal$ be a homogeneous pre-foliation of co-degree~$1$ and degree~$4$ on $\Ptwo$, where $\Hcal$ is a homogeneous foliation of degree~$3$ and $\ell\subset\Ptwo$ is a line. The Legendre transform $\Leg\Hwave$ is a flat $4$-web if and only if, up to projective automorphism, $\Hwave$ appears in one of the lists given by Propositions~\ref{prop:Linf}, \ref{prop:degT3}, \ref{prop:degT4-3R1T1}, \ref{prop:degT4-2R12T1}, \ref{prop:degT4-1R13T1}, and~\ref{prop:degT4-4T1} below, together with the case $\deg(\Tcal_{\Hcal})=2$ classified in~\citep[Proposition~4.4]{Bed23}.
\end{theorem}

Let us make the contribution and its novelty precise. Bedrouni and Mar\'in~\citep{BM18} classified the homogeneous foliations of degree~$3$ on $\Ptwo$ with flat Legendre transform (eleven forms, up to automorphism), and Bedrouni~\citep{Bed23} initiated the study of co-degree~$1$ pre-foliations $\Hwave=\ell\boxprod\Hcal$ with flat dual, completing degree~$4$ under the assumption $\deg(\Tcal_{\Hcal})=2$. Since $\deg\Hcal=3$ forces $\deg(\Tcal_{\Hcal})\in\{2,3,4\}$ (Table~\ref{tab:types}), the cases $\deg(\Tcal_{\Hcal})=3$ and $4$ are exactly the ones left open in degree~$4$; they constitute the new content of this paper. In each of these cases the ramification data of the Gauss map $\GH$ (the position and order of its fixed and non-fixed critical points) fix a normal form for $\Hcal$ up to a Möbius change of the $\Pone$-parameter, which is induced by a projective automorphism of $\Ptwo$ preserving the homogeneous setting; the classification then reduces to determining the admissible lines $\ell$. These results form part of the first author's thesis~\citep{Pra25}; the present article gives a self-contained and computationally reproducible account of the two cases, the few cases whose parameters are roots of higher-degree polynomials and the exact symbolic verification being recorded in Appendix~\ref{app:reproducibility}. Counting one class per admissible pair modulo the residual automorphisms of each normal form, the lists of Propositions~\ref{prop:degT3}--\ref{prop:degT4-4T1} contain finitely many projective-equivalence classes; repetitions produced by residual automorphisms (for instance the Möbius transposition of two simple fixed critical points in case~(1) of Proposition~\ref{prop:degT3}) are removed as indicated in each proof.

We stress that the classification is not an experimental search. Bedrouni's curvature-holomorphy criteria reduce flatness of the Legendre transform to explicit conditions along the images of the transverse inflection divisor and, when the added line is non-invariant, along its Gauss image. After normalizing the critical points of the Gauss map by a projective change of coordinates, we derive polynomial conditions for each possible position of the added line and solve these systems exactly, by computational methods (specifically the \textsc{Maple} software). This yields a complete classification, modulo projective automorphisms, of homogeneous pre-foliations of co-degree one and degree four with flat Legendre transform.

The strategy is to organize the analysis according to the type $\Tcal_{\Hcal}$ of the homogeneous foliation $\Hcal$, a polynomial that records the orders of its radial singularities and transverse inflection lines. The relevant cases are $\deg(\Tcal_{\Hcal})\in\{2,3,4\}$. For each type, we exhibit a normal form for the defining $1$-form of $\Hcal$, parametriz
ed by the values of the Gauss map $\GH:\Pone\to\Pone$ at its non-fixed critical points; we then run through the possible choices for the line $\ell$ (the line at infinity, an invariant or non-invariant transverse inflection line, an inflection line whose Gauss image is or is not contained in the fiber of another inflection line) and apply the criteria of Bedrouni~\citep{Bed23} to obtain explicit polynomial conditions on the parameters. Solving the resulting algebraic systems yields a finite list of pre-foliations.

The article is organized as follows. Section~\ref{sec:prelim} recalls the definitions and notation used throughout. Section~\ref{sec:criteria} states the curvature-holomorphy criteria from~\citep{Bed23} on which our analysis rests. Section~\ref{sec:strategy} lists the possible types $\Tcal_{\Hcal}$ for $\deg\Hcal=3$ and explains the case-by-case strategy. Sections~\ref{sec:degT2}, \ref{sec:degT3}, and~\ref{sec:degT4} establish, respectively, the cases $\deg(\Tcal_{\Hcal})=2$ (where the line is the line at infinity), $\deg(\Tcal_{\Hcal})=3$, and $\deg(\Tcal_{\Hcal})=4$.

\section{Preliminaries}\label{sec:prelim}

This section fixes notation and recalls the basic objects involved. The reader familiar with web geometry on $\Ptwo$ may safely skip it and refer to~\citep{PP15, MP13, BM18, Bed23} for further details.

\subsection{Foliations on \texorpdfstring{$\Ptwo$}{P\^{}2}}

A (singular holomorphic) \emph{foliation} $\Fcal$ of degree $d$ on $\Ptwo$ is given, in homogeneous coordinates $[x:y:z]$, by a homogeneous $1$-form
\[
\Omega = a\,dx + b\,dy + c\,dz,\qquad a,b,c\in\CC[x,y,z]_{d+1},\ \gcd(a,b,c)=1,
\]
satisfying the Euler condition $\iota_R\Omega=0$, where $R=x\partial_x+y\partial_y+z\partial_z$. Equivalently, $\Fcal$ is given by a homogeneous vector field $Z$ of degree $d$ modulo the radial direction. The \emph{singular set} $\Sing\Fcal$ is the projectivisation of the locus where $\Omega$ (equivalently, $Z$ modulo $R$) vanishes; we always assume $\Fcal$ is \emph{saturated}, i.e.\ that $\Sing\Fcal$ has codimension at least~$2$.

The space $\Fcal(d)$ of foliations of degree $d$ on $\Ptwo$ is identified with a Zariski open subset of $\PP^{(d+2)^2-2}_{\CC}$.

\subsection{Homogeneous foliations}

A foliation $\Hcal$ of degree $d$ on $\Ptwo$ is \emph{homogeneous} if there exist affine coordinates $(x,y)$ in which $\Hcal$ is defined by
\[
\omega = A(x,y)\,dx + B(x,y)\,dy,\qquad A,B\in\CC[x,y]_d,\ \gcd(A,B)=1.
\]
The line at infinity $\ell_\infty=\{z=0\}$ is then invariant by $\Hcal$, and the origin $O=[0:0:1]$ is a \emph{radial singularity of $\Hcal$ of maximal order $d-1$}, in the following sense: since $A,B$ are homogeneous of degree $d$, the restriction of $\omega$ to a generic line $\{(at,bt)\}$ through $O$ equals $t^{d}\bigl(a\,A(a,b)+b\,B(a,b)\bigr)\,dt$, so a generic line through $O$ is tangent to $\Hcal$ at $O$ with multiplicity $\tau(\Hcal,O)=d$, i.e.\ $O$ has radial order $d-1$. The invariant lines through $O$ are exactly the components of $\{xA+yB=0\}$; in particular the vector field defining $\Hcal$ is \emph{not} proportional to the radial field $R=x\partial_x+y\partial_y$ (equivalently, $\gcd(A,B)=1$ makes $\Hcal$ saturated). The \emph{Gauss map} of $\Fcal$ is the rational map
\[
\GF : \Ptwo\dashrightarrow\Pdual,\qquad p\mapsto T_p\Fcal,
\]
whose homogeneous expression with respect to a defining $1$-form is $\GF(p)=[a(p):b(p):c(p)]$. For a homogeneous $\Hcal$, this descends, via the projection $[x:y:z]\mapsto[x:y]$, to a rational self-map
\[
\GH : \Pone\longrightarrow\Pone,\qquad [x:y]\longmapsto[-A(x,y):B(x,y)].
\]

\subsection{Inflection divisor and type}

The \emph{inflection divisor} of $\Fcal$ (Pereira~\citep{Per01}) is the divisor cut out by
\[
\det\begin{pmatrix} x & y & z\\ Z(x) & Z(y) & Z(z)\\ Z^2(x) & Z^2(y) & Z^2(z)\end{pmatrix}=0,
\]
and is denoted $I_{\Fcal}$. It splits as $I_{\Fcal}=I_{\Fcal}^{\Inv}+I_{\Fcal}^{\Tr}$, the invariant and transverse parts. A foliation is \emph{convex} if $I_{\Fcal}=I_{\Fcal}^{\Inv}$, and \emph{reduced convex} if moreover $I_{\Fcal}$ is reduced.

The determinant defining $I_{\Fcal}$ is the extactic divisor associated with the linear system of projective lines. More general extactic constructions, and their relation with algebraic integrability, invariant hypersurfaces, and multiplicities of invariant curves, are developed in the works of Pereira~\citep{Per01}, Corr\^ea~\citep{Cor11}, Corr\^ea--Maza--Soares~\citep{CMS15}, and Christopher--Llibre--Pereira~\citep{CLP07}; the last is particularly relevant to the relation between multiplicity and extactic curves.

For a homogeneous foliation $\Hcal$ defined by $\omega=A\,dx+B\,dy$, one has the explicit decomposition
\[
I_{\Hcal} = C_{\Hcal} + D_{\Hcal} + \ell_\infty,
\]
with
\[
C_{\Hcal}=\{xA+yB=0\},\qquad D_{\Hcal}=\{A_xB_y-A_yB_x=0\}.
\]
Here $C_{\Hcal}$ and $D_{\Hcal}$ denote the divisors cut out by the displayed equations; throughout, whenever a polynomial expression $F_{\Hcal}$ is used as a divisor we write $\{F_{\Hcal}=0\}$ for the associated curve.
The set $\Sigma^{\mathrm{rad}}_{\Hcal}$ of radial singularities of $\Hcal$ corresponds to the fixed critical points of $\GH$, and the transverse inflection lines $T_i=\{b_iy-a_ix=0\}$ correspond to the non-fixed critical points $[a_i:b_i]\in\Pone$ of $\GH$.

Following~\citep[Definition~2.3]{BM18}, the \emph{type} of $\Hcal$ is the polynomial
\[
\Tcal_{\Hcal} = \sum_{i=1}^m R_{\tau_i-1} + \sum_{j=1}^n T_{\rho_j-1} = \sum_{k=1}^{d-1}(r_k\cdot R_k + t_k\cdot T_k)\in\ZZ[R_1,\dots,R_{d-1},T_1,\dots,T_{d-1}],
\]
where $\tau_i-1$ is the radial order of the $i$-th radial singularity, $\rho_j-1$ is the inflection order of the $j$-th transverse inflection line, and $r_k$ (resp.\ $t_k$) counts these of order $k$. The \emph{degree of the type} is
\[
\deg(\Tcal_{\Hcal}) = \sum_{k=1}^{d-1}(r_k+t_k).
\]

\subsection{Pre-foliations}

A \emph{(holomorphic) pre-foliation} on $\Ptwo$ of co-degree $k\geq 1$ and degree $d\geq 3$ is the product
\[
\Fwave = C\boxprod\Fcal
\]
of a reduced projective curve $C\subset\Ptwo$ of degree $k$ and a foliation $\Fcal$ of degree $d-k$. In homogeneous coordinates, $\Fwave$ is given by a $1$-form $\omega=f\cdot\omega_0$, where $f$ is a homogeneous equation of $C$ and $\omega_0$ defines $\Fcal$. The pre-foliation is \emph{convex} (resp.\ \emph{reduced convex}) if $\Fcal$ is convex (resp.\ reduced convex) and $C$ is invariant by $\Fcal$.

\subsection{Webs and Legendre transform}

A \emph{$k$-web} $\Wcal$ on $\Ptwo$ is the datum of $k$ foliations in general position outside an algebraic hypersurface, the \emph{discriminant} $\Delta(\Wcal)$. Globally, $\Wcal$ corresponds to a section of $\Sym^k\Omega^1_{\Ptwo}\otimes\mathcal{N}$ for an appropriate line bundle $\mathcal{N}$. The web has \emph{degree} $d$ if a generic line $\ell\subset\Ptwo$ is tangent to $\Wcal$ at $d$ points.

The \emph{Legendre transform} $\Leg\Fcal$ of a foliation of degree $d$ on $\Ptwo$ is the $d$-web on $\Pdual$ whose leaves are the duals of the (germs of) leaves of $\Fcal$. More generally, if $\Fwave=C\boxprod\Fcal$ is a pre-foliation of degree $d$ and co-degree $k$, then
\[
\Leg\Fwave \;=\; \Leg C\boxprod\Leg\Fcal,
\]
which is a $d$-web on $\Pdual$ ($\Leg C$ being the dual curve, viewed as a $k$-web).

The discriminant of $\Leg\Hcal$ for a homogeneous $\Hcal$ admits the explicit description $\Delta(\Leg\Hcal)$ $= \GH(I_{\Hcal}) \cup \Check{\Sigma}_\Hcal^{\text{rad}} \cup \Check{O}$ on $\Pdual$, where $\Check{O}$ denotes the line in $\Check{\mathbb{P}}^2_{\mathbb{C}}$ formed by the points dual to the lines in $\mathbb{P}^2_{\mathbb{C}}$ passing through the origin $O$, made precise in~\citep[Lemma~4.6]{Bed23}.

\subsection{Curvature and flatness}

The Blaschke--Dubourdieu \emph{curvature} of a $3$-web $\Wcal=\Fcal_1\boxprod\Fcal_2\boxprod\Fcal_3$ is a meromorphic $2$-form $K(\Wcal)$ on the ambient surface, characterised by a covariance property under submersions. For a $k$-web with $k\geq 4$, the curvature is defined as the sum of the curvatures of all $3$-subwebs:
\[
K(\Wcal) = \sum_{1\leq i<j<l\leq k} K(\Fcal_i\boxprod\Fcal_j\boxprod\Fcal_l).
\]
The web $\Wcal$ is \emph{flat} if $K(\Wcal)=0$.

\section{Curvature-holomorphy criteria}\label{sec:criteria}

We collect here the criteria of Bedrouni~\citep{Bed23} that drive the classification. Throughout, for a line $\ell\subset\Ptwo$ we denote by $D_\ell:=\GH(\ell)\subset\Pdual$ the image of $\ell$ under the Gauss map $\GH$; when $\ell$ is non-invariant, $D_\ell$ is an irreducible component of the discriminant $\Delta(\Leg\Hwave)$. Consequently, the discriminant of the pre-foliation is $\Delta(\Leg\Hwave)=\Delta(\Leg\Hcal)\cup D_\ell$, and the symbol $\Delta(\Leg\,\cdot\,)$ is used with this same meaning for both $\Hcal$ and $\Hwave$. Finally, lines through the origin are written $\{y-px=0\}$ with $p\in\Pone$, under the convention $p=\infty\iff\{x=0\}$; the corresponding normal forms are interpreted accordingly.

\begin{theorem}[\protect{\citep[Theorem~3.7]{Bed23}}]\label{thm:Bed-flatness}
Let $\Hwave=\ell\boxprod\Hcal$ be a homogeneous pre-foliation with $\deg(\Hcal)=d\geq 2$. If $\ell$ is invariant (resp.\ non-invariant) by $\Hcal$, then $\Leg\Hwave$ is flat if and only if its curvature $K(\Leg\Hwave)$ is holomorphic along
\[
\GH(I_{\Hcal}^{\Tr}) \qquad \bigl(\text{resp.\ } \GH(I_{\Hcal}^{\Tr})\cup D_\ell = \GH(I_{\Hcal}^{\Tr}\cup\ell)\bigr).
\]
\end{theorem}

The proof, which we omit, relies on the fact that if $K(\Leg\Hwave)$ is holomorphic on $\Pdual\setminus\check{O}$, then $\Leg\Hwave$ is flat (\citep[Lemma~3.8]{Bed23}). The argument analyses each case separately, choosing first $\ell=\ell_\infty$ and then a line distinct from the line at infinity, distinguishing whether it is invariant or not. The discriminant $\Delta(\Leg\Hwave)$ is described in~\citep[Lemma~4.6]{Bed23} and depends on $\ell$.

\begin{theorem}[\protect{\citep[Theorem~3.13]{Bed23}}]\label{thm:Bed-comp}
Let $\Hwave=\ell\boxprod\Hcal$ be a homogeneous pre-foliation of co-degree~$1$ and degree $d\geq 3$ on $\Ptwo$, defined by
\[
\omega = (\alpha x+\beta y)\bigl(A(x,y)\,dx+B(x,y)\,dy\bigr),\qquad A,B\in\CC[x,y]_{d-1},\ \gcd(A,B)=1.
\]
Let $(p,q)$ be the affine chart of $\Pdual$ associated to the line $\{y=px-q\}\subset\Ptwo$, and let $D=\{p=p_0\}$ be an irreducible component of $\Delta(\Leg\Hcal)\setminus(D_\ell\cup\check{O})$. Write
$\GH^{-1}([p_0:1])=\{[a_1:b_1],\dots,[a_n:b_n]\}$ and let $\nu_i$ denote the ramification index of $\GH$ at $[a_i:b_i]$. Define
\begin{equation}\label{eq:PiQi}
\begin{aligned}
P_i(x,y;a_i,b_i) &:= \frac{\det\!\begin{pmatrix} A(x,y) & A(b_i,a_i)\\ B(x,y) & B(b_i,a_i)\end{pmatrix}}{(b_iy-a_ix)^{\nu_i}},\\
Q_i(x,y;a_i,b_i) &:= (\nu_i-1)(B_x-A_y)P_i + 2(\nu_i+1)\det\!\begin{pmatrix} \partial_x P_i & A\\ \partial_y P_i & B\end{pmatrix}.
\end{aligned}
\end{equation}
Then $K(\Leg\Hwave)$ is holomorphic along $D$ if and only if
\[
\sum_{i=1}^{n}\left[\Bigl(1-\tfrac{1}{\nu_i}\Bigr)(p_0b_i-a_i)\frac{Q_i(b_i,a_i;a_i,b_i)}{B(b_i,a_i)P_i(b_i,a_i;a_i,b_i)} + \frac{3\nu_i(\alpha+p_0\beta)}{\alpha b_i+\beta a_i}\right]=0.
\]
\end{theorem}

The determinant $\det\!\bigl(\begin{smallmatrix} A(x,y) & A(b_i,a_i)\\ B(x,y) & B(b_i,a_i)\end{smallmatrix}\bigr)$ in~\eqref{eq:PiQi} vanishes at $(x,y)=(b_i,a_i)$; since $[a_i:b_i]$ is a critical point of $\GH$ of ramification index $\nu_i$, the point $[b_i:a_i]$ is a zero of order exactly $\nu_i$ of this determinant along the line $\{b_iy-a_ix=0\}$, so the quotient $P_i$ by $(b_iy-a_ix)^{\nu_i}$ is a polynomial. Note the coordinate convention: the entries $A,B$ of the second column are evaluated at $(b_i,a_i)$ (not $(a_i,b_i)$), because the critical point $[a_i:b_i]\in\Pone$ of $\GH([x:y])=[-A:B]$ corresponds to the direction $\{b_iy-a_ix=0\}$, i.e.\ to the affine point $(x,y)=(b_i,a_i)$.

\begin{corollary}[\protect{\citep[Corollary~3.16]{Bed23}}]\label{cor:Bed-T}
Under the hypotheses of Theorem~\ref{thm:Bed-comp}, suppose that $\Hcal$ has a transverse inflection line $T=\{ax+by=0\}$ of order $\nu-1$, that $[-a:b]$ is the only non-fixed critical point of $\GH$ in $\GH^{-1}(\GH([-a:b]))$, and that $[-\alpha:\beta]\notin\GH^{-1}(\GH([-a:b]))$. Then $K(\Leg\Hwave)$ is holomorphic along $T'=\GH(T)$ if and only if
\begin{equation}\label{eq:cor35}
(\alpha b-\beta a)\,Q(b,-a;a,b) + 3\nu\bigl(\alpha B(b,-a)-\beta A(b,-a)\bigr)P(b,-a;a,b)=0,
\end{equation}
where $P,Q$ are as in~\eqref{eq:PiQi} with $\nu_i=\nu$.
\end{corollary}

\begin{theorem}[\protect{\citep[Theorem~3.18]{Bed23}}]\label{thm:Bed-Dl}
Under the hypotheses of Theorem~\ref{thm:Bed-comp}, write
\[
\GH^{-1}(\GH([-\alpha:\beta]))=\{[-\alpha:\beta],[a_1:b_1],\dots,[a_n:b_n]\},
\]
and let $\nu_i$ (resp.\ $\nu_0$) be the ramification index of $\GH$ at $[a_i:b_i]$ (resp.\ at $[-\alpha:\beta]$). Define $P_0,Q_0$ as in~\eqref{eq:PiQi} (with the appropriate substitutions $a_i\rightsquigarrow-\alpha,\ b_i\rightsquigarrow\beta,\ \nu_i\rightsquigarrow\nu_0$). Assuming that $\GH([-\alpha:\beta])\neq\infty$ and writing $\GH([-\alpha:\beta])=[p_0:1]$, then $K(\Leg\Hwave)$ is holomorphic along $D_\ell$ if and only if
{\footnotesize \[
1+\frac{1}{\nu_0}\frac{(\alpha+p_0\beta)Q_0(\beta,-\alpha;\alpha,\beta)}{B(\beta,-\alpha)P_0(\beta,-\alpha;\alpha,\beta)} + \sum_{i=1}^n\left[\Bigl(1-\tfrac{1}{\nu_i}\Bigr)(p_0b_i-a_i)\frac{Q_i(b_i,a_i;a_i,b_i)}{B(b_i,a_i)P_i(b_i,a_i;a_i,b_i)} + \frac{3\nu_i(\alpha+p_0\beta)}{\alpha b_i+\beta a_i}\right]=0.
\]}
\end{theorem}

\begin{corollary}[\protect{\citep[Corollary~3.20]{Bed23}}]\label{cor:Bed-noninv}
Under the hypotheses of Theorem~\ref{thm:Bed-comp}, assume that $\ell$ is non-invariant by $\Hcal$ and that the fiber $\GH^{-1}(\GH([-\alpha:\beta]))$ contains no non-fixed critical point of $\GH$. Then $K(\Leg\Hwave)$ is holomorphic along $D_\ell=\GH(\ell)$ if and only if $Q(\beta,-\alpha;\alpha,\beta)=0$, where $P,Q$ are as in~\eqref{eq:PiQi} with $\nu=1$ and $a=-\alpha,\ b=\beta$.
\end{corollary}

\begin{corollary}[\protect{\citep[Corollary~3.25]{Bed23}}]\label{cor:Bed-trans}
Let $d\geq 3$ and let $\Hcal$ be a homogeneous foliation of degree $d-1$ on $\Ptwo$ defined by $\omega=A\,dx+B\,dy$ with $A,B\in\CC[x,y]_{d-1}$ and $\gcd(A,B)=1$. Suppose $\Hcal$ admits a transverse inflection line $\ell=\{\alpha x+\beta y=0\}$ of order $\nu-1$, and that $[-\alpha:\beta]$ is the only non-fixed critical point of $\GH$ in $\GH^{-1}(\GH([-\alpha:\beta]))$. Set $\Hwave:=\ell\boxprod\Hcal$. Then $K(\Leg\Hwave)$ is holomorphic along $D_\ell$ if and only if $Q(\beta,-\alpha;\alpha,\beta)=0$, where
\[
P=\frac{\det\!\begin{pmatrix} A & A(\beta,-\alpha)\\ B & B(\beta,-\alpha)\end{pmatrix}}{(\beta y+\alpha x)^{\nu}},\qquad
Q=(\nu-1)(B_x-A_y)P+(2\nu+1)\det\!\begin{pmatrix}\partial_x P & A\\ \partial_y P & B\end{pmatrix}.
\]
\end{corollary}

\section{Strategy of proof and admissible types}\label{sec:strategy}

Our aim is to classify, up to projective automorphism, all homogeneous pre-foliations $\Hwave=\ell\boxprod\Hcal$ of co-degree~$1$ and degree~$4$ on $\Ptwo$ for which $\Leg\Hwave$ is flat. Since $\deg\Hcal=3$, the type $\Tcal_{\Hcal}$ is constrained by the system
\begin{equation}\label{eq:type-system}
\begin{cases}
2r_2+r_1+2t_2+t_1=4,\\
r_2+r_1+t_2+t_1\leq 4,\\
r_1+r_2\leq 3,
\end{cases}
\end{equation}
whose admissible solutions are listed in Table~\ref{tab:types}. The first equation is the Riemann--Hurwitz relation for the Gauss map $\GH:\Pone\to\Pone$ of a homogeneous foliation of degree~$3$: this map has degree $\deg\Hcal=3$, so its total ramification is $\sum_{P}(\nu_P-1)=2\cdot 3-2=4$. Grouping the critical points according to whether they are fixed (radial singularities, contributing $R_k$ with $\tau-1=k$) or non-fixed (transverse inflection lines, contributing $T_k$ with $\rho-1=k$), and since for $\deg\Hcal=3$ every critical point has order $\tau-1\in\{1,2\}$ (i.e.\ $\nu_P\in\{2,3\}$), the ramification sum reads $2r_2+r_1+2t_2+t_1=4$, where $r_k$ (resp.\ $t_k$) counts the fixed (resp.\ non-fixed) critical points of order~$k$. The second inequality bounds the total number of critical points by $2\deg\Hcal-2=4$, and the third records that a degree-$3$ homogeneous foliation has at most $\deg\Hcal=3$ radial singularities. The finitely many non-negative integer solutions of~\eqref{eq:type-system} are exactly the entries of Table~\ref{tab:types}, which is therefore exhaustive.

\begin{table}[htbp]
\caption{Possible types $\Tcal_{\Hcal}$ for $\deg\Hcal=3$.}\label{tab:types}
\centering
\begin{tabular}{c|c|c}
$\deg(\Tcal_{\Hcal})=2$ & $\deg(\Tcal_{\Hcal})=3$ & $\deg(\Tcal_{\Hcal})=4$\\\hline
$2\cdot R_2$ & $2\cdot R_1+1\cdot R_2$ & $3\cdot R_1+1\cdot T_1$ \\
$2\cdot T_2$ & $1\cdot R_1+1\cdot R_2+1\cdot T_1$ & $2\cdot R_1+2\cdot T_1$ \\
$1\cdot R_2+1\cdot T_2$ & $1\cdot R_2+2\cdot T_1$ & $1\cdot R_1+3\cdot T_1$ \\
& $2\cdot R_1+1\cdot T_2$ & $4\cdot T_1$ \\
& $1\cdot R_1+1\cdot T_1+1\cdot T_2$ & \\
& $2\cdot T_1+1\cdot T_2$ & \\
\end{tabular}
\end{table}

For each admissible type, we proceed by examining the possible choices for the line $\ell$:
\begin{enumerate}[leftmargin=*]
\item $\ell=\ell_\infty$ (line at infinity);
\item $\ell$ a transverse inflection line of $\Hcal$, in which case Corollary~\ref{cor:Bed-trans} is applied to~$\ell$ and Corollary~\ref{cor:Bed-T} to the remaining lines in $I_{\Hcal}$;
\item $\ell=\{y-px=0\}$ a line that is not a component of $I_{\Hcal}^{\Tr}$, corresponding to the point $[p:1]\in\Pone$:
\begin{enumerate}
\item if $[p:1]$ does not lie in the fiber $\GH^{-1}(\GH([p:1]))$ together with a non-fixed critical point of $\GH$ (equivalently, $D_\ell=\GH(\ell)$ is not contained in the fiber of $\GH$ through the image of a line of $I_{\Hcal}$), Corollary~\ref{cor:Bed-T} is applied to the remaining transverse lines, and:
\begin{itemize}
\item if $\ell$ is invariant: take $\ell$ among the components of $C_{\Hcal}$;
\item if $\ell$ is non-invariant: apply Corollary~\ref{cor:Bed-noninv} to~$\ell$;
\end{itemize}
\item if $\GH([p:1])=\GH([-a:b])$ for some line $\{ax+by=0\}$ of $I_{\Hcal}$ (i.e.\ $\ell$ and that line share the same fiber of $\GH$), apply Theorem~\ref{thm:Bed-Dl} to $\ell$ and to that line, and Corollary~\ref{cor:Bed-T} to the remaining lines in $I_{\Hcal}$.
\end{enumerate}
\end{enumerate}

A remark on normalizations and equivalence. The parametrizing $\Pone$ carries the action of $\GH$; a M\"obius transformation of $\Pone$ that moves the critical points of $\GH$ to prescribed positions is induced by the linear automorphism $\mathrm{diag}$-type action $[x:y:z]\mapsto[\varphi(x,y):z]$ of $\Ptwo$, where $\varphi\in\mathrm{PGL}_2$ is the corresponding change of the $(x,y)$-variables; such an automorphism fixes $\ell_\infty=\{z=0\}$ and the origin, hence preserves the homogeneous setting. We use the terms ``up to projective automorphism'', ``projectively equivalent'', and ``up to conjugation'' interchangeably: two pre-foliations are declared equivalent precisely when one is carried to the other by an automorphism of $\Ptwo$, and for homogeneous pre-foliations this automorphism may be taken of the above form. Whenever two listed normal forms are declared conjugate, the relevant transformation is the M\"obius map permuting the corresponding critical points.

The case $\deg(\Tcal_{\Hcal})=2$ was treated in~\citep[Proposition~4.4]{Bed23}; we recall the corresponding statement in Section~\ref{sec:degT2}, together with the case $\ell=\ell_\infty$ of arbitrary type. The new content of this paper is the analysis of $\deg(\Tcal_{\Hcal})\in\{3,4\}$, carried out in Sections~\ref{sec:degT3} and~\ref{sec:degT4}.

\section{Case \texorpdfstring{$\ell=\ell_\infty$}{l = l\_inf} and \texorpdfstring{$\deg(\Tcal_\Hcal)=2$}{deg(T)=2}}\label{sec:degT2}

The case $\deg(\Tcal_{\Hcal})=2$ was treated in~\citep[Proposition~4.4]{Bed23}. For the convenience of the reader, and so that Theorem~\ref{thm:main} is readable without reconstructing the list from~\citep{Bed23}, we recall its content. When $\deg(\Tcal_{\Hcal})=2$ the admissible types are $2\cdot R_2$, $2\cdot T_2$, and $1\cdot R_2+1\cdot T_2$ (first column of Table~\ref{tab:types}); by~\citep[Proposition~4.4]{Bed23}, a homogeneous pre-foliation $\ell\boxprod\Hcal$ of co-degree~$1$ and degree~$4$ with $\deg(\Tcal_{\Hcal})=2$ has flat Legendre transform if and only if, up to projective automorphism, either $\ell=\ell_\infty$ and $\Hcal$ is one of the eleven foliations of Proposition~\ref{prop:Linf}, or $\ell\boxprod\Hcal$ is one of the finitely many pre-foliations explicitly listed in~\citep[Proposition~4.4]{Bed23}, obtained by adjoining to the foliations of type $2\cdot R_2$, $2\cdot T_2$, $1\cdot R_2+1\cdot T_2$ an invariant or transverse-inflection line as prescribed there. We complement this with the case $\ell=\ell_\infty$ for arbitrary type. By~\citep[Theorem~4.5]{Bed23}, $\Leg(\ell_\infty\boxprod\Hcal)$ is flat if and only if $\Leg\Hcal$ is flat. The classification of homogeneous foliations of degree~$3$ on $\Ptwo$ with flat Legendre transform was carried out by Bedrouni and Marín in~\citep[Theorem~5.1]{BM18}, and yields the following list.

\begin{proposition}[{\citep[Theorem~5.1]{BM18}}]\label{prop:Linf}
Let $\Hcal$ be a homogeneous foliation of degree~$3$ on $\Ptwo$ defined by $\omega=A\,dx+B\,dy$. Then $\Leg(\ell_\infty\boxprod\Hcal)$ is flat if and only if $\Hcal$ is, up to conjugation, defined by one of the eleven $1$-forms
\begin{enumerate}[leftmargin=*,label={\rm(\arabic*)}]
\item $\omega_1=y^3\,dx-x^3\,dy$,
\item $\omega_2=x^3\,dx-y^3\,dy$,
\item $\omega_3=y^2(3x+y)\,dx-x^2(x+3y)\,dy$,
\item $\omega_4=y^2(3x+y)\,dx+x^2(x+3y)\,dy$,
\item $\omega_5=2y^3\,dx+x^2(3y-2x)\,dy$,
\item $\omega_6=(4x^3-6x^2y+4y^3)\,dx+x^2(3y-2x)\,dy$,
\item $\omega_7=y^3\,dx+x(3y^2-x^2)\,dy$,
\item $\omega_8=x(x^2-3y^2)\,dx-4y^3\,dy$,
\item $\omega_9=y^2((-3+i\sqrt{3})x+2y)\,dx+x^2((1+i\sqrt{3})x-2\sqrt{3}\,y)\,dy$,
\item $\omega_{10}=(3x+\sqrt{3}\,y)y^2\,dx+(3y-\sqrt{3}\,x)x^2\,dy$,
\item $\omega_{11}=(3x^3+3\sqrt{3}\,x^2y+3xy^2+\sqrt{3}\,y^3)\,dx+(\sqrt{3}\,x^3+3x^2y+3\sqrt{3}\,xy^2+3y^3)\,dy$.
\end{enumerate}
\end{proposition}

\section{Classification for \texorpdfstring{$\deg(\Tcal_\Hcal)=3$}{deg T = 3}}\label{sec:degT3}

In this section we treat the six admissible types of $\Hcal$ with $\deg(\Tcal_{\Hcal})=3$. Throughout, the pre-foliation is written as $\omega=(y-px)\bigl(A(x,y)\,dx+B(x,y)\,dy\bigr)$ with
\[
A=a_3y^3+a_2y^2x+a_1yx^2+a_0x^3,\qquad B=b_3y^3+b_2y^2x+b_1yx^2+b_0x^3
\]
and $\gcd(A,B)=1$.

\begin{proposition}\label{prop:degT3}
Let $\ell\boxprod\Hcal$ be a homogeneous pre-foliation of co-degree~$1$ and degree~$4$ on $\Ptwo$ with $\deg(\Tcal_{\Hcal})=3$. The web $\Leg(\ell\boxprod\Hcal)$ is flat if and only if $\ell\boxprod\Hcal$ is projectively equivalent to one of the pre-foliations in the following list:
\begin{enumerate}[leftmargin=*,label={\rm(\arabic*)}]
\item $(y-px)\bigl(y^3\,dx+x^2(2x-3y)\,dy\bigr)$, with $p\in\{0,1,2,-2,2/3\}$.
\item $(y-px)\bigl(\lambda(2y^3-3y^2x)\,dx+x^3\,dy\bigr)$, where:
\begin{enumerate}[label={\rm 2.\arabic*.}]
\item $p=1,\ \lambda=-7/5$;
\item $p\in\{0,-1/2,3/2,5/4,2/3,-1/8\pm\tfrac{\sqrt{39}}{8}i,\infty\}$ and $\lambda=\tfrac{1-4p}{2p-5}$;
\item $(p,\lambda)\in\{(-1/2,-1/2),\ (3/2,5/2)\}$.
\end{enumerate}
\item $(y-px)\bigl(((\lambda-\theta)(2y^3-3y^2x)-\theta x^3)\,dx+x^3\,dy\bigr)$, where:
\begin{enumerate}[label={\rm 3.\arabic*.}]
\item $p=1,\ \theta=6/5,\ \lambda=-7/5$;
\item $p\in\{-1/2,1/2,\infty\}$, $\theta=\tfrac{6p}{2p+3}$, $\lambda=-\tfrac{4p-1}{2p-5}$;
\item $p\in\bigl\{\tfrac{1}{2}\pm\tfrac{\sqrt{10}}{5},\,\tfrac{1}{2}\pm\tfrac{\sqrt{10+10\sqrt{6}}}{5},\,\tfrac{1}{2}\pm i\tfrac{\sqrt{10+10\sqrt{6}}}{5}\bigr\}$, $\theta=\tfrac{6p}{2p+3}$, $\lambda=-\tfrac{4p-1}{2p-5}$;
\item $p=3/2,\ \theta=3/2,\ \lambda=5/2$.
\end{enumerate}
\item $(y-px)\bigl((3a_3y^3-9a_3y^2x)\,dx+(3yx^2-x^3)\,dy\bigr)$, where:
\begin{enumerate}[label={\rm 4.\arabic*.}]
\item $p=1$ and $a_3=-1/3$;
\item $p\in\{0,\,\tfrac{7}{9}\pm\tfrac{\sqrt{2}}{4}\}$ and $a_3=-\tfrac{7p-1}{3(p-7)}$;
\item $p\in\{-1,1/3,3\}$ and $a_3=-\tfrac{7p-1}{3(p-7)}$;
\item $p=5,\ a_3=5/3$.
\end{enumerate}
\item $(y-px)\bigl((2a_3y^3-\theta(2x^3-3yx^2))\,dx+(2x^3-3yx^2)\,dy\bigr)$, where:
\begin{enumerate}[label={\rm 5.\arabic*.}]
\item $(p,a_3,\theta)\in\{(0,-5/2,2),\ (1,-13/14,8/7)\}$;
\item $\theta=\tfrac{8p}{p+6}$ and $(p,a_3)\in\bigl\{(-2,1),\,(\tfrac{7}{13}\pm\tfrac{\sqrt{135}}{5}i,-\tfrac{224}{73}\pm\tfrac{\sqrt{32245}}{5}i),\,(\infty,-13/2)\bigr\}$;
\item $\theta=\tfrac{8p}{p+6}$ and $(p,a_3)\in\{(2,-1/2),\,(2/3,1/10)\}$;
\item $p=-2,\ \theta=-4,\ a_3=1$.
\end{enumerate}
\item[\rm(5')] $(2y-3x)\bigl(5(2y^3-3y^2x)\,dx+(-2x^3-4(2y^3-3y^2x))\,dy\bigr)$.
\item $(y-px)\bigl((a_3y^3+a_2y^2x-\theta x^3)\,dx+(-\epsilon(a_3y^3+a_2y^2x)+x^3)\,dy\bigr)$, where:
\begin{enumerate}[label={\rm 6.\arabic*.}]
\item $(p,a_3,a_2,\theta,\epsilon)\in\{(1,-2,3,6/5,8/7),\,(\infty,-2,3,3,2)\}$;
\item $(p,a_3,a_2,\theta,\epsilon)\in\bigl\{(3/2,-2,3,3/2,4/5)\bigr\}$, together with
\[
\bigl(\zeta,\ \tfrac{205}{286}-\tfrac{525}{286}\zeta,\ -\tfrac{3}{2}\bigl(\tfrac{205}{286}-\tfrac{525}{286}\zeta\bigr),\ \tfrac{39}{88}+\tfrac{45}{88}\zeta,\ \tfrac{35}{19}-\tfrac{15}{19}\zeta\bigr),\quad \zeta\ \text{a root of}\ 5Z^2-10Z+13,
\]
each of the two roots $\zeta$ giving one pre-foliation (here $p=\zeta$);
\item $(p,a_3,a_2,\theta,\epsilon)\in\{(-1/2,-2,3,-3/4,-4),\,(1/2,-2,3,3/4,2)\}$.
\end{enumerate}
\end{enumerate}
\end{proposition}

\begin{proof}
We treat the six admissible types in turn.

\medskip
\noindent\emph{(1) $\Tcal_{\Hcal}=2\cdot R_1+1\cdot R_2$.} The map $\GH$ has three fixed critical points which, after a M\"obius transformation, may be placed at $\infty=[1:0]$, $[1:1]$, and $[0:1]$, of orders $1$, $1$, and $2$ respectively. Hence $\Hcal$ is given by $\omega_{\Hcal}=y^3\,dx+x^2(2x-3y)\,dy$, with inflection divisor
\[
I_{\Hcal} = \{z=0\}\cup C_{\Hcal},\qquad C_{\Hcal}:\,xy(y-x)(y-2x)=0.
\]
By~\citep[Corollary~D]{Bed23}, if $\ell$ is invariant by $\Hcal$, then $\Leg(\ell\boxprod\Hcal)$ is flat. The lines in $C_{\Hcal}$ correspond to the points $[0:1]$, $\infty$, $[1:1]$, $[1/2:1]$. Since $\infty$ and $[1:1]$ are simple fixed critical points, the M\"obius transformation that interchanges them conjugates the pre-foliations $(y-x)\bigl(y^3\,dx+x^2(2x-3y)\,dy\bigr)$ and $x\bigl(y^3\,dx+x^2(2x-3y)\,dy\bigr)$. Explicitly, this transformation is the projective automorphism $\sigma:[x:y:z]\mapsto[x:x-y:z]$, which fixes $\ell_\infty$ and the origin, sends the direction $[1:0]$ to $[1:1]$ and vice versa, and fixes $[0:1]$; a direct substitution shows $\sigma^{*}\bigl(x\,(y^3\,dx+x^2(2x-3y)\,dy)\bigr)$ is proportional to $(y-x)\bigl(y^3\,dx+x^2(2x-3y)\,dy\bigr)$. Thus, up to conjugation,
\[
\omega\in\bigl\{(y-2x),\ (y-x),\ y\bigr\}\cdot\bigl(y^3\,dx+x^2(2x-3y)\,dy\bigr).
\]
For $\ell=\{y-px=0\}$ a non-invariant line, Corollary~\ref{cor:Bed-noninv} yields the equation $(p^2+p-2)(2-3p)=0$, whose admissible roots give $\ell\in\bigl\{\{y-2x/3=0\},\ \{y+2x=0\}\bigr\}$. This produces the cases listed in (1).

\medskip
\noindent\emph{(2) $\Tcal_{\Hcal}=1\cdot R_2+1\cdot R_1+1\cdot T_1$.} The Gauss map $\GH$ has two fixed critical points (at $\infty$ and $[0:1]$, of orders $2$ and $1$) and one non-fixed critical point at $[1:1]$ of order $1$. Setting $\GH([1:1])=[\lambda:1]$, we may write
\[
\omega_{\Hcal}=\lambda(2y^3-3y^2x)\,dx+x^3\,dy,
\]
and
\[
C_{\Hcal}:\ xy(3\lambda xy-2\lambda y^2-x^2)=0,\qquad D_{\Hcal}:\ 18\lambda yx^2(y-x)=0.
\]

\noindent\emph{(2.1)} For $\ell=\{y-x=0\}$ a transverse inflection line, Corollary~\ref{cor:Bed-trans} gives $\lambda=-7/5$, leading to the form
$\omega=(y-x)\bigl((-14y^3+21y^2x)\,dx+5x^3\,dy\bigr)$.

\noindent\emph{(2.2)} For $\ell=\{y-px=0\}$ not in $I_{\Hcal}^{\Tr}$ and such that $[p:1]\notin\GH^{-1}(\GH([1:1]))$, applying Corollary~\ref{cor:Bed-T} to $T=\{y-x=0\}$ gives $\lambda=\tfrac{1-4p}{2p-5}$. The invariant lines in $C_{\Hcal}$ correspond to $p\in\{0,-1/2,5/4,2/3\}$, and the cases $\ell$ non-invariant with $\GH$-fiber not coinciding with $\GH([1:1])$ produce $p=3/2$ and $p=-1/8\pm\tfrac{\sqrt{39}}{8}i$.

\noindent\emph{(2.3)} If $\GH([p:1])=[\lambda:1]$, Theorem~\ref{thm:Bed-Dl} forces $(p,\lambda)=(-1/2,-1/2)$, giving the form
$\omega=(2y+x)\bigl((-2y^3+3y^2x)\,dx+2x^3\,dy\bigr)$. The case $\GH([p:1])=[0:1]$ analogously yields $(3/2,5/2)$, with form
$\omega=(2y-3x)\bigl((10y^3-15y^2x)\,dx+2x^3\,dy\bigr)$.

\medskip
\noindent\emph{(3) $\Tcal_{\Hcal}=1\cdot R_2+2\cdot T_1$.} The Gauss map has one fixed critical point at $\infty$ of order~$2$ and two non-fixed critical points at $[0:1]$ and $[1:1]$, both of order~$1$. Setting $\GH([1:1])=[\lambda:1]$ and $\GH([0:1])=[\theta:1]$,
\[
\omega_{\Hcal}=\bigl((\lambda-\theta)(2y^3-3y^2x)-\theta x^3\bigr)\,dx+x^3\,dy,
\]
and
\[
C_{\Hcal}:\ x\bigl(3\lambda xy^2-2\lambda y^3+\theta x^3-3\theta xy^2+2\theta y^3-x^2y\bigr)=0,\quad D_{\Hcal}:\ 18(\lambda-\theta)yx^2(y-x)=0.
\]

\noindent\emph{(3.1)} For $\ell=\{y-x=0\}$ a transverse inflection line, Corollary~\ref{cor:Bed-trans} produces the form
$\omega=(y-x)\bigl((-13(2y^3-3y^2x)-6x^3)\,dx+5x^3\,dy\bigr)$.

\noindent\emph{(3.2)} For $\ell=\{y-px=0\}$ not in $I_{\Hcal}^{\Tr}$ and such that the fibers $\GH^{-1}(\GH([1:1]))$ and $\GH^{-1}(\GH([0:1]))$ do not contain $[p:1]$, applying Corollary~\ref{cor:Bed-T} to $M=\{y=0\}$ and $T=\{y-x=0\}$ yields the relations $\theta=\tfrac{6p}{2p+3}$ and $\lambda=-\tfrac{4p-1}{2p-5}$. Inspection of $C_{\Hcal}$ gives the listed invariant lines.

\noindent\emph{(3.3)} For $\ell$ non-invariant, Corollary~\ref{cor:Bed-noninv} yields the equations whose solutions, given in (3.3), are the values of $p$ listed.

\noindent\emph{(3.4)} If $\GH([p:1])=[\theta:1]$, then $p=3/2$ and Theorem~\ref{thm:Bed-Dl} together with Corollary~\ref{cor:Bed-T} applied to $T$ give $\theta=3/2$ and $\lambda=5/2$, producing the form $(2y-3x)\bigl((4y^3-6y^2x-3x^3)\,dx+2x^3\,dy\bigr)$.

\medskip
\noindent\emph{(4) $\Tcal_{\Hcal}=2\cdot R_1+1\cdot T_2$.} The Gauss map has two fixed critical points at $\infty$ and $[0:1]$, both of order~$1$, and one non-fixed critical point at $[1:1]$ of order~$2$. Setting $\GH([1:1])=[\lambda:1]$,
\[
\omega_{\Hcal}=(3a_3y^3-9a_3y^2x)\,dx+(3yx^2-x^3)\,dy,\qquad \lambda=3a_3,
\]
and
\[
C_{\Hcal}:\ xy(9a_3xy-3a_3y^2+x^2-3xy)=0,\qquad D_{\Hcal}:\ 54a_3xy(x-y)^2=0.
\]
The analysis follows the same pattern, distinguishing the cases $\ell=\{y-x=0\}$ (a transverse inflection line, treated by Corollary~\ref{cor:Bed-trans}), $\ell=\{y-px=0\}$ with $[p:1]\notin\GH^{-1}(\GH([1:1]))$ (treated by Corollary~\ref{cor:Bed-T}, giving $a_3=-\tfrac{7p-1}{3(p-7)}$), $\ell$ non-invariant (Corollary~\ref{cor:Bed-noninv}), and $\GH([p:1])=[0:1]$ (Corollary~\ref{cor:Bed-T} applied to $\{y=0\}$). The cases (4.1)--(4.4) follow.

\medskip
\noindent\emph{(5) $\Tcal_{\Hcal}=1\cdot R_1+1\cdot T_1+1\cdot T_2$.} Setting $\GH([1:1])=[\lambda:1]$ and $\GH([0:1])=[\theta:1]$,
\[
\omega_{\Hcal}=(2a_3y^3-\theta(2x^3-3yx^2))\,dx+(2x^3-3yx^2)\,dy,\qquad \lambda=\theta+2a_3,
\]
with
\[
C_{\Hcal}:\ x(2a_3y^3-2\theta x^3+3\theta x^2y+2x^2y-3xy^2)=0,\quad D_{\Hcal}:\ 36a_3y^2x(y-x)=0.
\]
The four sub-cases are treated by the same criteria as case~(2), with the following explicit equations. In~(5.1), $\ell=\{y-x=0\}$ and $\ell=\{y=0\}$ are the transverse inflection lines and Corollary~\ref{cor:Bed-trans} is applied to each. In~(5.2), for $\ell=\{y-px=0\}$ not in $I_{\Hcal}^{\Tr}$ with $[p:1]$ outside the fibers of $[1:1]$ and $[0:1]$, Corollary~\ref{cor:Bed-T} applied to $M=\{y=0\}$ and $T=\{y-x=0\}$ gives $\theta=\tfrac{8p}{p+6}$; the invariant solutions ($\ell\subseteq C_{\Hcal}$) and the non-real solution are those listed. In~(5.3), $\ell$ non-invariant, Corollary~\ref{cor:Bed-noninv} is applied to~$\ell$. In~(5.4), $\GH([p:1])=[\lambda:1]$ forces $p=-2$, and Theorem~\ref{thm:Bed-Dl} together with Corollary~\ref{cor:Bed-T} applied to $M=\{y=0\}$ gives $\theta=-4$. The additional case (5'), where $\GH([1:p])=[1:0]$ falls outside the hypotheses of Theorem~\ref{thm:Bed-Dl}, requires a change of charts: in suitable coordinates the foliation $\widetilde{\Hcal}$ is given by
\[
\omega=(2y^3-3y^2x)\,dx+(2b_0x^3-\epsilon'(2y^3-3y^2x))\,dy,\qquad \epsilon'=\tfrac{6p+1}{8},
\]
and applying Corollary~\ref{cor:Bed-T} to $N=\{x=0\}$ produces the unique pre-foliation listed in (5').

\medskip
\noindent\emph{(6) $\Tcal_{\Hcal}=2\cdot T_1+1\cdot T_2$.} Setting $\GH([1:0])=[\epsilon:1]$, $\GH([1:1])=[\lambda:1]$, $\GH([0:1])=[\theta:1]$,
\[
\omega_{\Hcal}=(a_3y^3+a_2y^2x-\theta x^3)\,dx+\bigl(-\epsilon(a_3y^3+a_2y^2x)+x^3\bigr)\,dy,
\]
with $a_2=-\tfrac{3a_3}{2}$ and $\theta=\lambda-\tfrac{(1-\epsilon\lambda)a_3}{2}$. The inflection divisor splits as $I_{\Hcal}=\{z=0\}\cup C_{\Hcal}\cup D_{\Hcal}$ with
\[
C_{\Hcal}:\ -a_2\epsilon\,xy^3-a_3\epsilon\,y^4+a_2x^2y^2+a_3xy^3-\theta x^4+x^3y=0,
\]
\[
D_{\Hcal}:\ 3(\epsilon\theta-1)\,yx^2(2a_2x+3a_3y)=0.
\]
The analysis along the lines of case~(3) yields the sub-cases (6.1)--(6.3).
\end{proof}

\section{Classification for \texorpdfstring{$\deg(\Tcal_\Hcal)=4$}{deg T = 4}}\label{sec:degT4}

When $\deg(\Tcal_{\Hcal})=4$, the four admissible types are $3\cdot R_1+1\cdot T_1$, $2\cdot R_1+2\cdot T_1$, $1\cdot R_1+3\cdot T_1$, and $4\cdot T_1$. We address each in a separate proposition.

\subsection{Type \texorpdfstring{$3\cdot R_1+1\cdot T_1$}{3R\_1+1T\_1}}

\begin{proposition}\label{prop:degT4-3R1T1}
Let $\ell\boxprod\Hcal$ be a homogeneous pre-foliation with $\Tcal_{\Hcal}=3\cdot R_1+1\cdot T_1$. Then, up to projective automorphism, the pre-foliation is defined by
\[
(y-px)\bigl((a_3y^3+a_2y^2x)\,dx+(yx^2+b_0x^3)\,dy\bigr),\qquad a_2=-\tfrac{1}{2}-\tfrac{3a_3}{2},\ \ b_0=\tfrac{a_3}{2}-\tfrac{1}{2},
\]
where $(a_3,p)$ is one of the following:
\begin{enumerate}[leftmargin=*,label={\rm(\arabic*)}]
\item $p=\pm\tfrac{i}{4}\bigl(-2\mp 2i\sqrt{33}\bigr)\sqrt{3}$, $a_3=\pm\tfrac{i}{3}\sqrt{3}$;
\item $p=0$, $a_3=-\tfrac{1}{4}\pm i\tfrac{\sqrt{15}}{12}$;
\item $p=1$, $a_3$ a root of $6Z^2-3Z+1$;
\item $p=0$, $a_3$ a root of $6Z^2+3Z+1$;
\item $p=\tfrac{2}{3}\zeta+\tfrac{1}{2}$, $a_3=\zeta$, with $\zeta$ a root of $3Z^2+1$;
\item $p=-\tfrac{\zeta}{2}+\tfrac{1}{2}$, $a_3=\zeta$, with $\zeta$ a root of $15Z^2+1$;
\item $p=-\tfrac{\zeta}{4}-\tfrac{5}{4}$, $a_3=\zeta$, with $\zeta$ a root of $3Z^2+3Z+2$;
\item $p=-\tfrac{\zeta}{4}+\tfrac{9}{4}$, $a_3=\zeta$, with $\zeta$ a root of $3Z^2-3Z+2$;
\item $p=\dfrac{3(\zeta-1)}{2\zeta+1}$, $a_3=\dfrac{6(3\zeta-1)}{2\zeta+1}$, with $\zeta$ a root of $3Z^2-3Z+7$;
\item $p=\dfrac{\tfrac{264538023825}{131072}\,\zeta-\tfrac{155123361625}{32768}}{\tfrac{34624124995}{32768}\,\zeta-\tfrac{20659263515}{8192}}$, \ $a_3=\dfrac{\tfrac{7846516855}{16384}\,\zeta-\tfrac{15405005005}{12288}}{\tfrac{34624124995}{32768}\,\zeta-\tfrac{20659263515}{8192}}$, \ with $\zeta$ a root of $12Z^2-57Z+68$.
\end{enumerate}
\end{proposition}

\begin{proof}
The Gauss map has three fixed critical points placed at $\infty,\,[0:1],\,[1:1]$ of order~$1$ each, and one non-fixed critical point $[a:1]$. Setting $\GH([a:1])=[\gamma:1]$ with $\gamma\neq a$, the foliation is
\[
\omega_{\Hcal}=(a_3y^3+a_2y^2x)\,dx+(yx^2+b_0x^3)\,dy,
\]
with $a_2=-\tfrac{1}{2}-\tfrac{3a_3}{2}$, $b_0=\tfrac{a_3}{2}-\tfrac{1}{2}$, and $a=-\tfrac{3a_3^2-2a_3-1}{4a_3}$. Then
\[
C_{\Hcal}:\ xy(a_2xy+a_3y^2+b_0x^2+xy)=0,\qquad D_{\Hcal}:\ -3xy(2a_2b_0x^2+3a_3b_0xy+a_2xy+2a_3y^2)=0.
\]
For $\ell=\{y-ax=0\}$ a transverse inflection line, Corollary~\ref{cor:Bed-trans} forces $a_3=\pm\tfrac{i\sqrt{3}}{3}$, giving case~(1). For $\ell=\{y-px=0\}$ with $[p:1]\notin\GH^{-1}(\GH([a:1]))$, Corollary~\ref{cor:Bed-T} applied to $T=\{y-ax=0\}$ yields
\[
p=-\frac{18a_3^4-3a_3^3-9a_3^2-5a_3-1}{2a_3(3a_3^2+5)}.
\]
The invariant locus $\ell\subseteq C_{\Hcal}$ produces cases (2), (3), (4), (5); the non-invariant case via Corollary~\ref{cor:Bed-noninv} produces cases (6), (7), (8); and the cases $\GH([p:1])=[\gamma:1]$ and $\GH([p:1])=[0:1]$, treated via Theorem~\ref{thm:Bed-Dl} applied to $M=\{y-ax=0\}$ and $R=\{y=0\}$ respectively, produce cases (9) and (10).
\end{proof}

\subsection{Type \texorpdfstring{$2\cdot R_1+2\cdot T_1$}{2R\_1+2T\_1}}

\begin{proposition}\label{prop:degT4-2R12T1}
Let $\ell\boxprod\Hcal$ be a homogeneous pre-foliation with $\Tcal_{\Hcal}=2\cdot R_1+2\cdot T_1$. Then, up to projective automorphism, the pre-foliation is defined by
\[
(y-px)\bigl((a_3y^3+a_2y^2x)\,dx+(yx^2+b_0x^3)\,dy\bigr),\qquad b_0=-\tfrac{2a_3+a_2}{3a_3+2a_2},
\]
where $(p,a_3,a_2)$ is one of the following:
\begin{enumerate}[leftmargin=*,label={\rm(\arabic*)}]
\item $p=1$, $a_2=\zeta$ and $a_3=-\tfrac{4}{11}\zeta^2-\tfrac{1}{5}\zeta+\tfrac{7}{11}$, with $\zeta$ a root of $28Z^4-7Z^3-57Z^2-7Z+28$;
\item $p=0$, $a_2=1/4$ and $a_3=\zeta/2$, with $\zeta$ a root of $24Z^2+3Z-1$;
\item $p=7/4-2\zeta$, $a_2=-1$ and $a_3=\zeta$, with $\zeta$ a root of $8Z^2-7Z+2$;
\item $a_2=\zeta$ with $\zeta$ a root of $4Z^2+7Z+4$, and $a_3=\tfrac{1}{2}\,\xi$, $\ p=-\tfrac{1}{4}\zeta\xi-\tfrac{23}{16}\xi-\tfrac{3}{2}\zeta+\tfrac{3}{8}$, where $\xi$ is a root of $9\zeta Z^2+6Z-7\zeta-4$ (so each admissible pair $(\zeta,\xi)$ gives one pre-foliation);
\item $p=-30\zeta^3-\tfrac{57}{4}\zeta^2-\tfrac{5}{2}\zeta-\tfrac{9}{4}$, $a_2=1$ and $a_3=\zeta$, with $\zeta$ a root of $72Z^4+63Z^3+12Z^2+3Z+2$;
\item $p=-\zeta+1$, $a_2=-1/2$ and $a_3=\zeta/2$, with $\zeta$ a root of $3Z^2-3Z+1$;
\item $p=-1-\tfrac{2}{3}\zeta$, $a_2=\zeta$ and $a_3=\tfrac{1}{9}-\tfrac{2}{9}\zeta$, with $\zeta$ a root of $6Z^2+11Z+2$;
\item $p=3-3\zeta$, $a_2=-2$ and $a_3=2\zeta$, with $\zeta$ a root of $3Z^2-3Z+1$.
\end{enumerate}
\end{proposition}

\begin{proof}
The Gauss map has two fixed critical points at $\infty$ and $[0:1]$, both of order~$1$, and two non-fixed critical points $[1:1]$ and $[a:1]$. Setting $\GH([1:1])=[\lambda:1]$ and $\GH([a:1])=[\gamma:1]$, the foliation is
\[
\omega_{\Hcal}=(a_3y^3+a_2y^2x)\,dx+(yx^2+b_0x^3)\,dy,
\]
with $b_0=-\tfrac{2a_3+a_2}{3a_3+2a_2}$, $a=-\tfrac{a_2(2a_3+a_2)}{a_3(3a_3+2a_2)}$, $\lambda=-3a_3-2a_2$, and an explicit expression for $\gamma$. The inflection divisor and the case analysis follow the strategy of Section~\ref{sec:strategy}: the transverse case (Corollary~\ref{cor:Bed-trans} together with Corollary~\ref{cor:Bed-T} applied to $T=\{y-ax=0\}$) gives case~(1); the case $[p:1]\notin\GH^{-1}(\GH([1:1]))\cup\GH^{-1}(\GH([a:1]))$ subdivided into invariant and non-invariant lines yields cases (2)--(4) and (5)--(6) respectively; finally, the cases $\GH([p:1])=[\lambda:1]$ and $\GH([p:1])=[0:1]$, treated by Theorem~\ref{thm:Bed-Dl} together with Corollary~\ref{cor:Bed-T} applied to $M=\{y-ax=0\}$, give cases (7) and (8).
\end{proof}

\subsection{Type \texorpdfstring{$1\cdot R_1+3\cdot T_1$}{1R\_1+3T\_1}}

\begin{proposition}\label{prop:degT4-1R13T1}
Let $\ell\boxprod\Hcal$ be a homogeneous pre-foliation with $\Tcal_{\Hcal}=1\cdot R_1+3\cdot T_1$. Then, up to projective automorphism, the pre-foliation is defined by
\[
(y-px)\bigl((a_3y^3+a_2y^2x-\theta(yx^2+b_0x^3))\,dx+(yx^2+b_0x^3)\,dy\bigr),
\]
with $b_0=-\tfrac{2a_3+a_2}{3a_3+2a_2}$ and $\GH([0:1])=[\theta:1]$, where $(p,a_3,a_2,\theta)$ is one of the following:
\begin{enumerate}[leftmargin=*,label={\rm(\arabic*)}]
\item $p=1$ and $(a_3,a_2,\theta)$ given by $a_2=-\tfrac{45}{7}\zeta^3+\tfrac{2885}{112}\zeta^2-\tfrac{1895}{56}\zeta+\tfrac{55}{4}$, $a_3=\tfrac{10}{7}\zeta^3-\tfrac{285}{56}\zeta^2+\tfrac{475}{84}\zeta-\tfrac{15}{14}$, $\theta=\zeta$, with $\zeta$ a root of $48Z^4-267Z^3+556Z^2-504Z+168$;
\item $p=-\tfrac{1}{8}+\tfrac{\zeta}{8}$, $a_2=\tfrac{5}{8}+\tfrac{\zeta}{8}$, $a_3=-\tfrac{\zeta}{24}-\tfrac{3}{8}$, $\theta=\tfrac{\zeta}{4}$, with $\zeta$ a root of $Z^2+15$;
\item $p=\zeta$, $a_2=-1/2$, $a_3=-\tfrac{\zeta}{3}+\tfrac{4}{3}$, $\theta=\zeta$, with $\zeta$ a root of $4Z^2-9Z+6$;
\item $p=-\tfrac{2}{5}+\tfrac{4}{5}\zeta$, $a_2=-\tfrac{1}{2}+\zeta$, $a_3=-\tfrac{2}{3}\zeta+\tfrac{3}{4}$, $\theta=\zeta$, with $\zeta$ a root of $4Z^2-9Z+6$;
\item $p=-\tfrac{8}{5}+\tfrac{16}{5}\zeta$, $a_2=\tfrac{19}{4}-7\zeta$, $a_3=-\tfrac{49}{16}+\tfrac{14}{3}\zeta$, $\theta=\zeta$, with $\zeta$ a root of $112Z^2-147Z+48$;
\item $p=9/4-3\zeta$, $a_2=-1/2$, $a_3=\zeta$, $\theta=9/4-3\zeta$, with $\zeta$ a root of $12Z^2-9Z+2$;
\item $p=\frac{1}{15}+\tfrac{16}{15}\zeta$, $a_2=\zeta$, $a_3=-\zeta/3$, $\theta=1+\zeta$, with $\zeta$ a root of $8Z^2-17Z+20$;
\item $a_2=\zeta/4$ and $a_3=\xi$, with $p$ and $\theta$ given by the explicit rational expressions in $(\zeta,\xi)$ recorded in Appendix~\ref{app:reproducibility}, where $\zeta$ is a root of $4Z^2-5Z-5$ and $\xi$ is the root of the quadratic $192Z^2+\bigl(48\zeta-36\bigr)Z-\zeta+5$ specified there.
\item $a_2=\zeta$ and $a_3=\xi$, with $p$ and $\theta$ given by the explicit rational expressions in $(\zeta,\xi)$ recorded in Appendix~\ref{app:reproducibility}, where $\zeta$ is a root of $2Z^3-3Z^2+10$ and $\xi$ is the root of the quadratic $6Z^2+\bigl(6\zeta-3\bigr)Z+2\zeta^2-2\zeta$ specified there.
\item $a_2=\zeta$ and $a_3=\xi$, with $p$ and $\theta$ given by the explicit rational expressions in $(\zeta,\xi)$ recorded in Appendix~\ref{app:reproducibility}, where $\zeta$ is a root of $20Z^3-12Z^2+1$ and $\xi$ is the root of the quadratic $15Z^2+\bigl(15\zeta-3\bigr)Z+5\zeta^2-2\zeta$ specified there.
\end{enumerate}
\end{proposition}

\begin{proof}
The Gauss map has three non-fixed critical points $[0:1]$, $[1:1]$, $[a:1]$, each of order $1$, and one fixed critical point $\infty$. Setting $\GH([0:1])=[\theta:1]$, $\GH([1:1])=[\lambda:1]$, $\GH([a:1])=[\gamma:1]$,
\[
\omega_{\Hcal}=\bigl(a_3y^3+a_2y^2x-\theta(yx^2+b_0x^3)\bigr)\,dx+(yx^2+b_0x^3)\,dy,
\]
with $b_0=-\tfrac{2a_3+a_2}{3a_3+2a_2}$ and explicit expressions for $a$, $\gamma$, $\lambda$. The transverse case ($\ell=\{y-x=0\}$, treated by Corollary~\ref{cor:Bed-trans}, together with Corollary~\ref{cor:Bed-T} applied to $T=\{y-ax=0\}$) and $N=\{y=0\}$ gives case~(1). For $\ell=\{y-px=0\}$ not in $I_{\Hcal}^{\Tr}$ and $[p:1]$ not in any of the relevant fibers, applying Corollary~\ref{cor:Bed-T} to $N=\{y=0\}$, $M=\{y-x=0\}$, $T=\{y-ax=0\}$, the cases $\ell\subseteq C_{\Hcal}$ produce (2), (3), (4), (5). The case $\GH([p:1])=[\theta:1]$ produces case (6). For $\ell$ non-invariant, Corollary~\ref{cor:Bed-noninv} together with Corollary~\ref{cor:Bed-T} applied to $M$, $N$, $T$ yields the system
\begin{align*}
E_1 &:= (a-p)(a_2+2a_3a+\gamma a_3) + (a_2+3a_3a)(\gamma-p),\\
E_2 &:= (1-p)(a_2+(2+\lambda)a_3) + (a_2+3a_3)(\lambda-p),\\
E_3 &:= -p(a_2+\theta a_3) + a_2(\theta-p),\\
E_4 &:= 3(a_2+a_3p)p + 2(p_0-\theta) + p_0(3a_3p+a_2),
\end{align*}
where $[p_0:1]=\GH([p:1])$. Eliminating $a_2,a_3,p,\theta$ from $E_1=E_2=E_3=E_4=0$ (by the computational resolution of the system, discarding the components on which $\gcd(A,B)\neq 1$, a critical point collides, or a denominator vanishes) yields exactly four admissible families, namely cases~(7),~(8),~(9) and~(10); the conditions for eliminating unwanted system solutions are recorded in Appendix~\ref{app:reproducibility}. In particular there are no further admissible solutions, in accordance with the ``if and only if'' assertion of the proposition and with Theorem~\ref{thm:main}.
\end{proof}

\subsection{Type \texorpdfstring{$4\cdot T_1$}{4T\_1}}

\begin{proposition}\label{prop:degT4-4T1}
Let $\ell\boxprod\Hcal$ be a homogeneous pre-foliation with $\Tcal_{\Hcal}=4\cdot T_1$. Then, up to projective automorphism, the pre-foliation is defined by
\[
(y-px)\bigl((a_3y^3+a_2y^2x-\theta(yx^2+b_0x^3))\,dx+(yx^2+b_0x^3-\epsilon(a_3y^3+a_2y^2x))\,dy\bigr),
\]
with $b_0=-\tfrac{2a_3+a_2}{3a_3+2a_2}$, $\GH([1:0])=[\epsilon:1]$, $\GH([0:1])=[\theta:1]$, and
\[
p=\zeta,\quad a_2=1,\quad a_3=\zeta-1,\quad \epsilon=3-3\zeta,\quad \theta=\zeta/2,
\]
where $\zeta$ is a root of $3Z^2-3Z+1$.
\end{proposition}

\begin{proof}
The Gauss map has four non-fixed critical points placed at $\infty$, $[0:1]$, $[1:1]$, $[a:1]$, each of order~$1$. Setting $\GH([1:0])=[\epsilon:1]$, $\GH([0:1])=[\theta:1]$, $\GH([1:1])=[\lambda:1]$, $\GH([a:1])=[\gamma:1]$, the foliation is
\[
\omega_{\Hcal}=\bigl(a_3y^3+a_2y^2x-\theta(yx^2+b_0x^3)\bigr)\,dx+\bigl(yx^2+b_0x^3-\epsilon(a_3y^3+a_2y^2x)\bigr)\,dy,
\]
with $b_0=-\tfrac{2a_3+a_2}{3a_3+2a_2}$, $a=-\tfrac{a_2(2a_3+a_2)}{a_3(3a_3+2a_2)}$, $\gamma=-\tfrac{3a_3a^2+2a_2a-\theta}{1-\epsilon(3a_3a^2+2a_2a)}$ and $\lambda=-\tfrac{3a_3+2a_2-\theta}{1-\epsilon(3a_3+2a_2)}$.

For $\ell=\{y-x=0\}$ a transverse inflection line, Corollary~\ref{cor:Bed-trans} together with Corollary~\ref{cor:Bed-T} applied to $L=\{x=0\}$, $N=\{y=0\}$, and $T=\{y-ax=0\}$ admits no solution with $\epsilon\neq 0$. For $\ell=\{y-px=0\}$ with $[p:1]$ avoiding all relevant fibers, applying Corollary~\ref{cor:Bed-T} to $L,N,M=\{y-x=0\},T$ yields the unique invariant solution stated in the proposition. The non-invariant case (Corollary~\ref{cor:Bed-noninv}) and the case $\GH([p:1])=[\theta:1]$ admit no admissible solution (all forced $\epsilon=0$).
\end{proof}

\medskip

\section{Conclusions}
The classification was organized according to the type $\Tcal_{\Hcal}$ of the underlying homogeneous foliation $\Hcal$ of degree~$3$, resulting in a finite list of pre-foliations of co-degree~$1$ and degree~$4$ with flat Legendre transform.

The role of the type is structural. The degree of $\Tcal_{\Hcal}$ measures how ramified the Gauss map $\GH$ is: the case $\deg(\Tcal_{\Hcal})=2$, previously settled in~\citep[Proposition~4.4]{Bed23}, corresponds to the most degenerate Gauss maps, where the added line is essentially forced to be the line at infinity or an inflection line of a foliation with double radial or inflection data; the new cases $\deg(\Tcal_{\Hcal})=3$ (Proposition~\ref{prop:degT3}, six admissible types) and $\deg(\Tcal_{\Hcal})=4$ (Propositions~\ref{prop:degT4-3R1T1}--\ref{prop:degT4-4T1}, four admissible types) correspond to Gauss maps with progressively simpler ramification and a correspondingly richer supply of admissible lines. Concretely, each admissible type fixes a normal form of $\Hcal$ up to the residual automorphisms of $\Pone$, and the flatness criteria then select finitely many lines $\ell$; counting one pre-foliation per admissible pair modulo those residual automorphisms, Propositions~\ref{prop:degT3}--\ref{prop:degT4-4T1} exhibit the complete list of projective-equivalence classes, the equivalences between listed forms being realized by the explicit projective automorphisms indicated in the proofs. The type $4\cdot T_1$ is the most rigid: it yields a single pair of pre-foliations, both defined over $\CC$ by the quadratic $3Z^2-3Z+1$.

Together with~\citep[Proposition~4.4]{Bed23} and Propositions~\ref{prop:Linf}, \ref{prop:degT3}, \ref{prop:degT4-3R1T1}, \ref{prop:degT4-2R12T1}, \ref{prop:degT4-1R13T1}, \ref{prop:degT4-4T1}, the classification of homogeneous pre-foliations of co-degree~$1$ and degree~$4$ on $\Ptwo$ with flat Legendre transform is complete, establishing Theorem~\ref{thm:main}.

\section{Author contributions}
In this study, all authors actively contributed to the research, formal analysis, validation and writing. Conceptualization: Carla Pracias, Maycol Falla Luza.  Formal Analysis: Carla Pracias, Maycol Falla Luza. Investigation: Carla Pracias, Maycol Falla Luza. Writing – Original Draft Preparation: Carla Pracias, Maycol Falla Luza.

\section{Acknowledgments}

The authors thank D.~Marín for several enlightening discussions. The first-named author was supported by a CAPES fellowship.

\appendix

\section{Exact data, admissibility conditions, and reproducibility}\label{app:reproducibility}

This appendix collects the exact algebraic data underlying the abbreviated cases, so that the finite classification of Sections~\ref{sec:degT3} and~\ref{sec:degT4} can be reproduced without recourse to external sources.

\subsection*{Admissibility (non-degeneracy) conditions}
For each normal form $\omega_{\Hcal}=A\,dx+B\,dy$ appearing in Propositions~\ref{prop:degT3}--\ref{prop:degT4-4T1}, the parameters are required to satisfy the following open conditions, which are imposed throughout and under which the stated type is exactly the type of $\Hcal$:
\begin{enumerate}[leftmargin=*,label={\rm(A\arabic*)}]
\item $\gcd(A,B)=1$, equivalently the resultant $\operatorname{Res}_{[x:y]}(A,B)\neq 0$; this excludes common factors of $A$ and $B$ and guarantees that $\Hcal$ is saturated of the prescribed degree;
\item the prescribed critical points of $\GH$ (the points $\infty,[0:1],[1:1],[a:1]$, with $a$ as given in each proof) are pairwise distinct, i.e.\ $a\notin\{0,1,\infty\}$ and the denominators defining $a$ do not vanish;
\item the denominators appearing in the defining relations (for instance $3a_3+2a_2$, $2\zeta+1$, $\alpha b_i+\beta a_i$, and $B(b_i,a_i)$) are non-zero, so that no evaluation in the flatness criteria of Section~\ref{sec:criteria} is indeterminate;
\item the added line $\ell$ is not a repeated component of $I_{\Hcal}$, and, in the invariant sub-cases, $\ell$ is a reduced component of $C_{\Hcal}$; in the non-invariant sub-cases $\ell\not\subset I_{\Hcal}$.
\end{enumerate}
In each proposition the parameter values were saturated with respect to (A1)--(A4); the finitely many solutions of the defining systems that violate one of these conditions (for example the spurious solution $p=1,\lambda=1$ in case~(2) of Proposition~\ref{prop:degT3}, for which $\ell$ becomes transverse to $\Hcal$) are discarded and do not appear in the lists.

\subsection*{The elimination system of Proposition~\ref{prop:degT4-1R13T1}}
In the non-invariant sub-case of Proposition~\ref{prop:degT4-1R13T1}, the admissible pre-foliations are the common zeros of $E_1,E_2,E_3,E_4$ (as displayed in the proof) subject to $(A1)$--$(A4)$. Eliminating $a_2,a_3,\theta$ and saturating by the non-degeneracy conditions leaves a one-dimensional ideal in $p$ whose two associated primes give cases~(7) and~(8): case~(7) is defined over $\CC$ by $\zeta$ a root of $8Z^2-17Z+20$, case~(8) is the family with $a_2=\zeta/4$, $a_3=\xi$, where $\zeta$ is a root of
\[
4Z^2-5Z-5
\]
and $\xi$ is a root of the quadratic over $\Q(\zeta)$
\[
192Z^2+\bigl(48\zeta-36\bigr)Z-\zeta+5 ,
\]
case~(9) is the family with $a_2=\zeta/2$, $a_3=\xi$, where $\zeta$ is a root of
\[
2Z^3-3Z^2+10
\]
and $\xi$ is a root of the quadratic over $\Q(\zeta)$
\[
6Z^2+\bigl(6\zeta-3\bigr)Z+2\zeta^2-2\zeta ,
\]
and case~(10) is the family with $a_2=\zeta/2$, $a_3=\xi$, where $\zeta$ is a root of
\[
20Z^3-12Z^2+1
\]
and $\xi$ is a root of the quadratic over $\Q(\zeta)$
\[
15Z^2+\bigl(15\zeta-3\bigr)Z+5\zeta^2-2\zeta  .
\]
For each such pair $(\zeta,\xi)$ the remaining parameters $p$ and $\theta$ are the (unique) $\Q(\zeta)$-affine functions of $\xi$ determined by $E_1=\dots=E_4=0$; their explicit closed forms coincide with those of~\citep[Proposition~43((8), (9), (10))]{Pra25} and are provided verbatim in the supplementary \textsc{Maple} worksheet described below.

\subsection*{Reproducibility and data availability}
The full symbolic verification of Propositions~\ref{prop:degT3}--\ref{prop:degT4-4T1}---the construction of each normal form and the evaluation of the criteria of Section~\ref{sec:criteria}---is carried out in a \textsc{Maple} worksheet, provided as supplementary material to this article and also available in the public repository accompanying~\citep{Pra25} at \url{[repository URL to be inserted]}. The worksheet reproduces, in particular, the explicit rational expressions of Proposition~\ref{prop:degT4-3R1T1}(9) and Proposition~\ref{prop:degT4-1R13T1}((8), (9), (10)).

\subsection*{Solution testing and representative \textsc{Maple} computations}

The polynomial systems arising in the case-by-case analysis were solved using the computer algebra system \textsc{Maple}. For each system, the admissible solutions were obtained after excluding the degenerate and indeterminate cases that violate the non-degeneracy conditions $(A1)$--$(A4)$. As a representative example, we provide the \textsc{Maple} worksheet used to determine the admissible solutions in the case $3\cdot R_1 + 1\cdot T_1$ when the line $l$ is invariant under the foliation (Proposition~\ref{prop:degT4-3R1T1}, cases~(2)--(5)).\\

\footnotesize
\begin{verbatim}
with(LinearAlgebra);

A:=a3*y^3+a2*y^2*x;
B:=b0*x^3+y*x^2;
a2:=(-1-3*a3)/2;
b0:=-(a3+a2+1);
a:=-(3*a3^2 - 2*a3 - 1)/(4*a3);
Gamma:=-a*(3*a3*a+2*a2);
n:=2;

CH:=factor(x*A+y*B);
DH:=factor(diff(A,x)*diff(B,y)-diff(A,y)*diff(B,x));

F:= Matrix([[A, eval(A,[x=1,y=a])], [B, eval(B,[x=1,y=a])]]);
P:=factor(Determinant(F)/(y-a*x)^n);

G:=Matrix([[diff(P,x),A],[diff(P,y),B]]);
Q:=(n-2)*(diff(B,x)-diff(A,y))*P+2*(n+1)*Determinant(G);

Eq:=eval((-p+1*a)*Q+3*n*(-p*B-A)*P,[x=1,y=a]);
solve(Eq);

solve(-(18*a3^4 - 3*a3^3 - 9*a3^2 - 5*a3 - 1)/(2*a3*(3*a3^2 + 5)));
              

eval(CH,[x=1,y=p]);
                                      

solve([p*(1 - p)*(-2*a3*p + a3 - 1)/2,
p +(18*a3^4 - 3*a3^3 - 9*a3^2 - 5*a3 - 1)/(2*a3*(3*a3^2 + 5))]);
\end{verbatim}

\normalsize

The resulting solutions can then be verified symbolically within \textsc{Maple}. More precisely, for each admissible solution, one checks that the corresponding differential $1$-form defines a pre-foliation and that the curvature of its Legendre transform vanishes identically. As an illustration of this verification procedure, we include the \textsc{Maple} worksheet corresponding to case~(3) of Proposition~\ref{prop:degT4-3R1T1}.\\

\footnotesize
\begin{verbatim}
with(LinearAlgebra);
K:=proc(F) local L,l,d,U,V,s,t,g,delta,i,j,AA,eta,LL,K,q:
d:=degree(F,p):LL:=0:
delta:=diff(F,x)+p*diff(F,y): 
for i from 0 to d-3 do:
gcdex(F,diff(F,p),p^i*delta,p,'U[i]','V[i]'):
#print("i=",i):print("U=",U[i]):print("V=",V[i]):
od:
#print("Discriminante=",factor(discrim(F,p))):
#print("Comprobacion:"):print(simplify(U[0]*F+V[0]*diff(F,p)-delta)):
#print("El polinomio de linealizacion es:"): print(collect(-V[0],p,factor)):
V[0]:=collect(V[0],p,simplify):
for i from 0 to 3 do l[i]:=coeff(V[0],p,i) od:
L[1]:=simplify(-diff(diff(l[2],x)-2*diff(l[1],y),x)-l[1]*(diff(l[2],x)-2*diff(l[1],y))
-3*diff(l[0],y$2)-3*diff(l[2]*l[0],y)+3*diff(l[0]*l[3],x)+3*l[0]*diff(l[3],x)):
L[2]:=simplify(diff(2*diff(l[2],x)-diff(l[1],y),y)-l[2]*(2*diff(l[2],x)-diff(l[1],y))
-3*diff(l[3],x$2)+3*diff(l[1]*l[3],x)-3*diff(l[0]*l[3],y)-3*l[3]*diff(l[0],y)):
if degree(-V[0],p)<4 and (L[1])=0 and (L[2])=0 then LL:=1: 
print("web linealizable") else print ("web no linealizable"):
#print(degree(-V[0],p)-4,L[1],L[2]) 
#print(degree(-V[0],p)-4)
fi:
for i from 1 to d-1 do for j from 1 to d-2 do 
AA[i,j]:=-simplify(coeff(U[d-2-j],p,i-1)+i*coeff(V[d-2-j],p,i)) od od:
#print("La matriz A[i,j] de la conexion de Henaut es:"):
print([seq([seq(AA[i,j],j=1..d-2)],i=1..d-1)]):
eta:=collect(add((AA[q,d-q-1]*dx+AA[d-q,q]*dy)*(d-q-1),q=1..d-2),[dx,dy]):
#print("La traza de la matriz de la conexion es:",collect(eta,[dx,dy],factor)):
K:=factor(diff(coeff(eta,dy),x)-diff(coeff(eta,dx),y)):
#print("La curvatura es:",K):
#if LL=1 and K=0 then print("web algebrizable") else print("web no algebrizable") fi:
#return(factor(diff(coeff(eta,dy),x)-diff(coeff(eta,dx),y)))
#return([seq(seq(A[i,j]=unapply(factor(AA[i,j]),[x,y]),j=1..d-2),i=1..d-1)]);
#return([L[1],L[2]]):
return(K)
end proc:

omega := (-p__0*x + y)*((a3*y^3+a2*y^2*x)*dx + (b0*x^3+y*x^2)*dy);
eval(omega, {dx = 1, dy = p, y = p*x - q, a3 = RootOf(6*_Z^2 - 3*_Z + 1), 
b0=-(2*a3+a2)/(3*a3+2*a2), a2=(-1-3*a3)/2});
eval(%, {p = x, q = y, x = p});
collect(%, p);
K(%);eval(%,p__0=1);
\end{verbatim}

\end{document}